\documentclass[11pt]{article}
\usepackage{amsmath}
\usepackage{amsfonts}
\usepackage{amssymb, amsthm}
\usepackage{color}
\usepackage[margin=1in]{geometry}
\usepackage{hyperref}
\usepackage{verbatim}
\usepackage{ulem}

\newtheorem{theorem}[equation]{Theorem}
\newtheorem{proposition}[equation]{Proposition}
\newtheorem{lemma}[equation]{Lemma}

\newtheorem{conjecture}[equation]{Conjecture}

\theoremstyle{definition}
\newtheorem{defn}[equation]{Definition}
\newtheorem{eg}[equation]{Example}
\newtheorem{rmk}[equation]{Remark}
\newenvironment{remark}[1][]{\begin{rmk}[#1]\pushQED{\qed}}{\popQED \end{rmk}}

\numberwithin{equation}{section}

\newcommand{\fh}{\mathfrak{h}}

\newcommand{\bZ}{\mathbf{Z}}
\newcommand{\GL}{\mathbf{GL}}
\newcommand{\cS}{\mathcal{S}}
\newcommand{\bF}{\mathbf{F}}
\DeclareMathOperator{\rank}{rank}
\DeclareMathOperator{\Sym}{Sym}

\newcommand{\ol}{\overline}

\newcommand{\arxiv}[1]{\href{http://arxiv.org/abs/#1}{{\tt arXiv:#1}}}

\title{Representations of Cherednik Algebras Associated to Symmetric and Dihedral Groups in Positive Characteristic}
\author{Carl Lian}

\begin{document}

\maketitle

\begin{abstract} We consider irreducible lowest-weight
  representations of Cherednik algebras associated to certain
  classes of complex reflection groups in characteristic $p$. In particular, we study
  maximal graded submodules of Verma modules associated to these
  algebras. Various results and conjectures are presented concerning
  generators of these maximal submodules, which are found by computing
  singular polynomials of Dunkl operators.
\end{abstract}

\section{Introduction}
\indent\indent Double affine Hecke algebras were first introduced by Cherednik in 1994 to study the Macdonald conjectures, which have since been proven. Cherednik algebras, certain degenerations of these algebras, were later studied in 2002 by Etingof--Ginzburg in \cite{calogeromoser}. Since then, the representation theory of Cherednik algebras has become a topic of study in itself. One of the main problems is to understand the dimensions, and in particular, the Hilbert series, of their lowest-weight irreducible representations. While this problem is very difficult to study in general, much is known about a large number of cases. For example, Hilbert series for lowest-weight irreducible representations of Cherednik algebras associated to $S_{n}$ in characteristic zero are calculated in \cite{gordon}.

However, the positive characteristic case is not as well-studied, though progress has recently been made in the case of rank 1 groups in \cite{latour}, as well the cases of the matrix groups $\text{GL}_{n}(\mathbf{F}_{q})$ and $\text{SL}_{n}(\mathbf{F}_{q})$ in \cite{chen}. Furthermore, the case of the symmetric group $S_{n}$ in characteristic $p$ with $p>n$ is studied geometrically in \cite{bfg}. The representation theory of Cherednik algebras in positive characteristic is of particular interest for several reasons. For example, the irreducible lowest-weight representations are always finite dimensional, a phenomenon which does not occur in characteristic zero. Also, many of the tools used to study the characteristic zero case do not carry over to positive characteristic, so new techniques are needed. Finally, because the structures of representations of (complex) reflection groups change in characteristic $p$, the resulting changes in the structures of representations of their associated Cherednik algebras are also of interest. 

In this paper, we consider representations of Cherednik algebras associated with complex reflection groups in characteristic $p$. In Section \ref{symmetricspecial}, we consider the algebra associated to the symmetric group $S_{3}$ and its trivial representation when the value of the parameter $c$ is taken to be in the field $\mathbf{F}_{p}$, the only case in which the Hilbert series of the irreducible lowest-weight representation differs from that of generic $c$. Following \cite{chmutova} and \cite{bo}, we show that the irreducible quotient $L_{c}=M_{c}/J_{c}$ is a complete intersection and give generators for $J_{c}$ for most values of $c$. In Section \ref{dihedral}, we consider the dihedral groups $G(m,m,2)$, giving a complete answer for one dimensional representations $\tau$ as well as certain two-dimensional representations $\tau$ when $m$ is odd. Finally, in Section \ref{symmetricgeneric}, we return to the group $S_{n}$ and its trivial representation, considering another special case: when $p|n$. Here, we give a recursive algorithm for constructing minimal degree generators of the ideal $J_{c}$, and present partial calculations of these generators.

\section{Definitions and Previous Results}\label{defs}

\begin{defn}
Let $\fh$ be a vector space, and let $s\in\GL(\fh)$ be a finite order operator on $\fh$. $s$ is a \textbf{reflection} if $\rank(1-s)=1$. A subgroup $G\subset\GL(\fh)$ generated by reflections is a \textbf{reflection group}.
\end{defn}

\begin{defn}
Let $G \subset \GL(\fh)$ be a reflection group where $\fh$ is a vector
space over a field $K$. Let $\cS$ be the set of reflections in
$G$. For each $s \in \cS$, we pick a vector $\alpha_s \in \fh^*$ that
spans the image of $1-s$, where we use the induced action of $G$ on $\fh^*$, and let $\alpha^\vee_s \in \fh$ be defined
by the property
\[
(1-s)x = (\alpha^\vee_s, x)\alpha_s.
\]

Let $\hbar$ be a parameter and $c_s$ be a parameter for each $s \in
\cS$, where we require that $c_s = c_{s'}$ if $s$ and $s'$ are
conjugate. Let $T(\fh \oplus \fh^*)$ be the tensor algebra. The {\bf
  Cherednik algebra} $H_{\hbar, c}(G,\fh)$ is the quotient of
$(K[G] \ltimes T(\fh \oplus \fh^*)) \otimes_K K(\hbar, \{c_s\})$ by the
relations
\[ [x,x'] = 0, \quad [y,y'] = 0, \quad [y,x] = \hbar(y,x) - \sum_{s
  \in \cS} c_s (y, \alpha_s) (x, \alpha_s^\vee)s
\]
where $x,x' \in \fh^*$ and $y,y' \in \fh$. Here, $(y,x)\in K$ denotes $x$ evaluated at $y$. and $[,]$ denotes the commutator in $H_{\hbar,c}$.

We take $\hbar=1$ throughout this paper; note that as long as $\hbar\neq0$ we have $H_{\hbar,c}\cong H_{1,c/\hbar}\cong H_{1,c}$.
\end{defn}

As in the case of characteristic zero (see \cite[Proposition 3.5]{etingofma}), we have a PBW basis for $H_{\hbar,c}$ of elements
\begin{equation*}
g\prod_{i=1}^{r}y_{i}^{m_{i}}\prod_{i=1}^{r}x_{i}^{n_{i}},
\end{equation*}
where $g\in G,y_{i}\in \fh,x_{i}\in\fh^{*}$. As the Dunkl representation is not faithful in positive characteristic, one needs a bit of care in reproving this result; on the other hand, we do have a faithful map from $H_{\hbar,c}$ to the localization of differential operators, so the proof from \cite{etingofma} carries over without issues.

\begin{defn}
Let $\tau$ be a representation of $G$. We let $\Sym(\fh)$
  act as 0 on $\tau$ and construct the \textbf{Verma module}
  \[
  M_{c}(G,\fh,\tau) = H_{\hbar,c}(G,\fh)
  \otimes_{K[G]\ltimes\Sym (\fh )}\tau.
  \]
These give lowest-weight representations of $H_{\hbar,c}$.
\end{defn}

\begin{theorem}[{\cite[3.12]{etingofma}}]
 Assume $\tau$ irreducible. Let $\ol{c} \colon S\rightarrow K$ be defined so that $\ol{c}_{s}=c_{s^{-1}}$. There exists a unique contravariant form $\beta_{c} \colon M_{c}(G,\fh,\tau)\times
  M_{\overline{c}}(G,\fh^{*},\tau^{*})\rightarrow K$ satisfying the following properties:
\begin{enumerate}
\item[i.] For all $x\in \fh^{*}$, we have $\beta_{c}(xf,h)=\beta_{c}(f,xh)$.
\item[ii.] For all $y\in \fh$, we have $\beta_{c}(f,yh)=\beta_{c}(yf,h)$.
\item[iii.] For all $v\in\tau$ and $w\in\tau^{*}$, we have $\beta_{c}(v,w)=w(v)$.
\end{enumerate}
Furthermore, $M_{c}$ has a unique maximal graded submodule $J_{c}$, which may be realized as the radical of the kernel of $\beta_{c}$.
\end{theorem}

\begin{defn}
Let $x\in\Sym(H^{*})$, and let $\{x_{1},x_{2},\ldots,x_{n}\}$ and $\{y_{1},y_{2},\ldots,y_{n}\}$ denote dual bases of $H^{*}$ and $H$, respectively. We define the \textbf{Dunkl operator} $D_{y_{i}}$ by
\begin{center}
$D_{y_{i}}(x)=\dfrac{\partial}{\partial x_{i}}(x)\otimes v-\displaystyle\sum_{s\in S}c_{s}\dfrac{2(y_{i},\alpha_{s})}{1-\lambda_{s}}\cdot\dfrac{x-s(x)}{\alpha_{s}}\otimes s(v)$,
\end{center}
\end{defn}

The Dunkl operator is the main tool in dealing with elements of $J_{c}$ for the following reason:

\begin{theorem}[{\cite[Proposition 3.2]{etingofma}}]
In $M_{c}$, $y_{i}x_{i_{1}}x_{i_{2}}\cdots x_{i_{m}}=D_{y_{i}}(x_{i_{1}}x_{i_{2}}\cdots x_{i_{m}})$, that is, the action of $y_{i}$ on $\Sym(H^{*})$ is given by the Dunkl operator.
\end{theorem}

\begin{rmk}
Note that $\beta(D_{k}P,Q)=\beta(P,y_{k}Q)$. Thus, if $D_{k}P=0$, we have $P\in J_{c}$: in this case $P$ is said to be \textbf{singular}. More generally, if $D_{k}P\in J_{c}$, then $P\in J_{c}$. However, if $P$ is a minimal degree element of $J_{c}$, we must have $D_{k}P=0$.
\end{rmk}

\begin{proposition} \label{prop:specialc}
  Given a representation $\tau$ of the symmetric
  group $\Sigma_n$, let $h^{\rm gen}(t)$ be the Hilbert function of
  $L_c(\tau)$ for $c$ generic. If $c \notin {\bf F}_p$, then we have
  $h^c(t) = h^{\rm gen}(t)$ where $h^c(t)$ is the Hilbert function of
  $L_c(\tau)$.
\end{proposition}

\begin{proof}
  Fix a $\bZ$-form of $\tau$, i.e., so all $g \in \Sigma_n$ act by
  integer-valued matrices. Let $R$ be the ring $\bZ[c]$ localized at
  the prime ideal generated by $p$. Hence $R$ is a principal ideal
  domain. For generic $c$, let $M_c(\tau)$ be the corresponding Verma
  module over the Cherednik algebra (defined over $R$). Let
  $\beta_c^d$ be a matrix representing the contravariant form on
  $M_c(\tau)$ in degree $d$. The determinant of $\beta_c^d$ depends on
  the $R$-bases chosen, but the number of times that $p$ divides the
  determinant does not, so let $o_{R,p}(\beta_c^d)$ be this number. We
  can characterize $o_{R,p}(\beta_c^d)$ in a different way. For each
  $i \ge 1$, let $K(i) = \{v \mid \beta_c^d v \text{ is divisible by }
  p^i \}$. Write $K(i)_p = K(i) \otimes_{\bZ} \bZ/p$. Then
  $o_{R,p}(\beta_c^d) = \sum_{i \ge 1} \dim_{{\bf F}_p} K(i)_p$, which
  can be shown by considering a Smith normal form of $\beta^c_d$ over
  $R$.

  If $c$ is transcendental over ${\bf F}_p$, the there is nothing to prove, so we assume that $c$ is algebraic over ${\bf F}_p$. Let $a(x) \in {\bf F}_p[x]$ be the minimal polynomial of $c$ over
  ${\bf F}_p$. Lift this to an integer polynomial
  $\tilde{a}(x) \in \bZ[x]$, i.e., the reduction modulo $p$ of
  $\tilde{a}(x)$ is $a(x)$. Let $\tilde{c}$ be a root of
  $\tilde{a}(x)$. Then $\bZ[x]/(\tilde{a}(x)) \otimes_{\bZ} \bZ/p$ is
  a finite extension ${\bf F}_q$ of ${\bf F}_p$ containing $c$. Let
  $R'$ be the localization of $\bZ[x] / (\tilde{a}(x))$ at the prime
  ideal generated by $p$.

  Consider the Verma module $M_{\tilde{c}}(\tau)$ over the Cherednik
  algebra defined over $R'$. Let $\beta^d_{\tilde{c}}$ be a matrix
  representing the contravariant form on $M_{\tilde{c}}(\tau)$ in
  degree $d$. We can define $o_{R',p}(\beta^d_{\tilde{c}})$ and
  $K'(i)$ as before. Note that
  \begin{align} \label{eqn:jantzenineq} 
    \dim_{{\bf F}_q} K'(i)_p \ge \dim_{{\bf F}_p} K(i)_p
  \end{align}
  and saying that the $d$th coefficient of $h^c(t)$ is strictly less
  than the $d$th coefficient of $h^{\rm gen}(t)$ is equivalent to saying
  that \eqref{eqn:jantzenineq} is a strict inequality for $i=1$, and
  hence equivalent to the strict inequality
  $o_{R',p}(\beta^d_{\tilde{c}}) > o_{R,p}(\beta^d_c)$. 

  So suppose that this inequality holds. The determinant of
  $\beta^d_c$ is of the form $N_d \prod_j (a_j c - b_j)$ where $N_d
  \in \bZ$ and $a_j, b_j \in \bZ$ are such that $a_j \ne 0$ and
  $\gcd(a_j,b_j) = 1$ \cite[Corollary 3.3]{etingof}, and hence
  $\det(\beta^d_{\tilde{c}}) = N_d \prod_j (a_j \tilde{c} -
  b_j)$. Hence there is some $j$ such that $p$ divides $(a_j \tilde{c}
  - b_j)$. Reducing this relation modulo $p$, this implies that $c \in
  {\bf F}_p$ since we cannot have both $a_j$ and $b_j$ divisible by
  $p$.
\end{proof}

\begin{remark} The above proof shows that in general, a function $c$
  is a special value (i.e. a value of $c$ for which the character is not equal to the generic character) in characteristic $p$ only if it is the reduction
  mod $p$ of a special value $c$ coming from characteristic 0. 
\end{remark}

\section{$S_{n}$ with special values of $c$}\label{symmetricspecial}
\indent\indent Throughout this section, we consider $M_{c}(S_{n},\fh,\tau)$ with $\tau$ trivial. Let $S_{n}$ act by permutation of indices on the basis $\{x_{1},x_{2},\ldots,x_{n}\}$. Note that reflections in $S_{n}$ are simply transpositions, which all have the same cycle type, requiring $c$ to be constant. We consider $c\in\mathbf{F}_{p}$, which are the only values of $c$ that can give different Hilbert series from $L_{c}(S_{n},\fh,\tau)$ where $c$ is taken to be generic, by \ref{prop:specialc}.

\begin{theorem}
When $c=0$, $h_{L_{c}}(t)=\left(\dfrac{1-t^{p}}{1-t}\right)^{n}$ for all $n$.
\end{theorem}

\begin{proof}
Because $c=0$, the Dunkl operators are partial derivatives, and thus kill $x_{1}^{p},\ldots,x_{n}^{p}$. Letting $J_{c}'$ be the ideal generated by these elements, we have $h_{M_{c}/J'_{c}}(t)=\left(\dfrac{1-t^{p}}{1-t}\right)^{n}$. Now, it is clear by induction that for any monomial $X$ with degree less than $p$ in each of the $x_{i}$, we have $\beta(X,X)\neq0$, so that the coefficient on the $t^{d}$ term of $h_{L_{c}}(t)$ is at least that of $\left(\dfrac{1-t^{p}}{1-t}\right)^{n}=h_{M_{c}/J'_{c}}(t)$. However, noting that $J_{c}'\subseteq J_{c}$, we immediately obtain $h_{L_{c}}(t)=\left(\dfrac{1-t^{p}}{1-t}\right)^{n}$ (and $J_{c}'=J_{c}$).
\end{proof}

\begin{theorem}
When $c=1/n$, we get $h_{L_{c}}(t)=\dfrac{1-t^{p}}{1-t}$.
\end{theorem}

\begin{proof}
Note that $D_{i}(x_{i})=1-c(n-1)=c$ and $D_{i}(x_{j})=c$ whenever $i\neq j$, so $x_{1}-x_{j}\in J_{c}$ for $j=2,3,\ldots,n$. Furthermore, $x_{1}^{p}\in J_{c}$. Also, for $0\le d\le p-1$, we may check by induction that $\beta(x_{1}^{d},x_{1}^{d})=0$. It follows that for $d<p$, the $d$-th graded component of $L_{c}$ is spanned by $x_{i}^{d}$, and that for $d\ge p$ the component is trivial, implying the result.
\end{proof}

We now turn our attention to the case in which $n=3$, which is of particular interest.

\begin{theorem}
When $n=3$ and $p>3$, express $c$ as a positive integer with $c<p$. In the following three cases, assuming the resulting Hilbert series agrees with the degrees of the generators given (see remark below), $M_{c}/J_{c}$ is a complete intersection, where the degrees of the generators of $J_{c}$ are noted below:
\begin{enumerate}
\item $0<c<p/3$: $p,p+3c,p+3c$
\item $p/3<c<p/2$: $3c-p,3c-p,p$
\item $2p/3<c<p$: $p-3c,p-3c,p$
\end{enumerate}
\end{theorem}

\begin{rmk}
We expect that the work of \cite{bo} will give formulas for the Hilbert series of $L_{c}$ in these cases, agreeing with those of the respective quotients of $M_{c}$ by polynomials of the above degrees.
\end{rmk}

\begin{proof}
Consider the polynomial $P(t)=(1-tx_{1})^{c'}(1-tx_{2})^{c'}(1-tx_{3})^{c'}$, and let $G$ be the coefficient of the $t^{3c'+1}$ term of $P$. Then, we first show that $\dfrac{\partial G}{\partial x_{1}},\dfrac{\partial G}{\partial x_{2}},\dfrac{\partial G}{\partial x_{3}}$ are killed by Dunkl operators in characteristic zero where we take $c=c'$. We will assume $3c'$, but $c'$ is not an integer (in particular $c'\neq0$): the Taylor series of each factor modulo $p$ will be defined since $p>3$.

Without loss of generality, consider $\dfrac{\partial G}{\partial x_{1}}$. Since partial derivatives with respect to the $x_{i}$ do nothing to $t$, it suffices to check that the $t^{3c'+1}$ coefficient of $\dfrac{\partial P}{\partial x_{1}}$ is killed by Dunkl operators. We have
\begin{center}
$\dfrac{\partial P}{\partial x_{1}}=-c't(1-tx_{1})^{c'-1}(1-tx_{2})^{c'}(1-tx_{3})^{c'}$.
\end{center}

Now, 
\begin{align*}
\dfrac{1}{-c't}D_{1}\dfrac{\partial G}{\partial x_{1}} =& -t(c'-1)(1-tx_{1})^{c'-2}(1-tx_{2})^{c'}(1-tx_{3})^{c'}\\
&-tc'(1-tx_{1})^{c'-1}(1-tx_{2})^{c'-1}(1-tx_{3})^{c'}-tc'(1-tx_{1})^{c'-1}(1-tx_{2})^{c'}(1-tx_{3})^{c'-1}.
\end{align*}

Dropping another factor of the indeterminate $t$ and changing sign, we wish to check that the $t^{3c'-1}$ coefficient of
\begin{center}
$(c'-1)(1-tx_{1})^{c'-2}(1-tx_{2})^{c'}(1-tx_{3})^{c'}$\\
$+c'(1-tx_{1})^{c'-1}(1-tx_{2})^{c'-1}(1-tx_{3})^{c'}+c'(1-tx_{1})^{c'-1}(1-tx_{2})^{c'}(1-tx_{3})^{c'-1}$
\end{center}
is zero. Consider the coefficient of $x_{1}^{i}x_{2}^{j}x_{3}^{k}t^{3c'-1}$, where we note that we must have $i+j+k=3c'-1$ (otherwise this coefficient is trivially zero). Up to a sign, this is
\begin{center}
$(c'-1)\dbinom{c'-2}{i}\dbinom{c'}{j}\dbinom{c'}{k}+c'\dbinom{c'-1}{i}\dbinom{c'-1}{j}\dbinom{c'}{k}+c'\dbinom{c'-1}{i}\dbinom{c'}{j}\dbinom{c'-1}{k}$,
\end{center}
which, upon multiplication by the appropriate factors becomes simply
\begin{center}
$(c'-1-i)+(c'-j)+(c'-k)=0$.
\end{center}

Now, in characteristic $p$, we construct singular polynomials of the desired degrees in the three cases listed above. In Case 1, consider $\dfrac{\partial G}{\partial x_{1}}$ and $\dfrac{\partial G}{\partial x_{2}}$, with $c'=c+p/3$, and in the other two cases, take the same polynomials with $c'=c-p/3$; it is trivial to check that these give the correct degrees. Scale these polynomials by the appropriate rational factor so that they are non-zero in characteristic $p$: it is now clear that the resulting polynomials $G_{1},G_{2}$ are singular in characteristic $p$. Furthermore, let $G_{3}=x_{1}^{p}+x_{2}^{p}+x_{3}^{p}$, which is both killed by partial derivatives and an $S_{3}$-invariant, so also singular.

Note that in each case, $G_{1},G_{2},G_{3}$ are linearly independent. Let $J'_c$ be the ideal that they generate. Then, we see that, using the fact that $M_{c}/J'_{c}$ is a complete intersection, $h_{M_{c}/J'_{c}}(t)$ agrees with $h_{L_{c}}(t)$ (as we expect to be computed in \cite{bo}, see above), indeed we must have $J'_{c}=J_{c}$. This completes the proof.
\end{proof}

\begin{rmk}
We conjecture that when $p/2<c<2p/3$, $M_{c}/J_{c}$ is again a complete intersection, with generators of degrees $6c-3p,p,p$. It is easy to check that the degree $6c-3p$ generator is $(x_{1}-x_{2})^{2c-p}(x_{2}-x_{3})^{2c-p}(x_{3}-x_{1})^{2c-p}$ and that one of the degree $p$ generators is $x_{1}^{p}+x_{2}^{p}+x_{3}^{p}$. Furthermore, it is known that in the case of $c=1/2$, the second degree $p$ generator is 
\begin{equation*}
\displaystyle\sum_{\Sym}\dfrac{x_{1}^{p}(x_{1}-x_{2})}{x_{1}-x_{3}};
\end{equation*}
however, the form of this third generator is unclear in general.
\end{rmk}

\begin{conjecture}
Consider $J_{c}(S_{n},\fh,\tau)$ with $\tau$ trivial and $n$ arbitrary, $p>n$, and $c\in\mathbf{F}_{p}$ expressed as an integer with $1\le c\le p-1$. Let $S$ denote the set of rationals numbers of the form $\frac{ap}{b}$ with integers $a,b\in[0,p)$. If no element of $S$ lies between $c$ and $c+1$, then the sum of the degrees of the elements of a minimal set of generators of $J_{c+1}(S_{n},\fh,\tau)$ is exactly $n!$ more than the analogous sum for $J_{c}(S_{n},\fh,\tau)$.
\end{conjecture}

\begin{rmk}
The above has been verified for $n=3$ and, in a small number of cases, $n=4$. One possible approach to this conjecture would be to consider some positive characterstic analogue of the shift functors of Berest-Chalykh, see \cite{bc}.
\end{rmk}

\begin{conjecture}
$M_{c}(S_{n},\fh,\tau)/J_{c}(S_{n},\fh,\tau)$ is a Gorenstein algebra for all $n$ whenever $\tau$ is trivial, both when $c\in\mathbf{F}_{p}$ and $c\notin\mathbf{F}_{p}$. More generally, for all $\tau$, $h_{L_{c}(S_{n},\fh,\tau)}(t)$ is palindromic.
\end{conjecture}

\begin{rmk}
In general, when $\tau$ is trivial, we do not expect $M_{c}/L_{c}$ to be a complete intersection. However, in the case that $n=3$, because $S_{3}$ has rank 2, the above conjecture immediately implies that $M_{c}/J_{c}$ is a complete intersection.
\end{rmk}

\section{Dihedral Groups $G(m,m,2)$}\label{dihedral}
\indent\indent In this section, we construct singular polynomials of Dunkl operators for the rank 2 dihedral group $D_{m}=G(m,m,2)$. Observe that we must have $p\nmid m$. Suppose our field $K$ contains all $m$th roots of unity, and let $\zeta$ be a fixed primitive $m$-th root of unity. The reflections in $D_{m}$ may be realized as acting by the $2 \times 2$ matrices
\begin{equation*}
\begin{bmatrix}
 &\zeta^{-k}\\
\zeta^{k}&  
\end{bmatrix}
\end{equation*}
for $0\le k<m$. When $m$ is odd, all such
reflections lie in the same conjugacy class. However, when $m$ is
even, there are two conjugacy classes or reflections, given by those
with $i$ odd and $i$ even.

\begin{theorem}
When $\tau$ is trivial and $c$ is generic, $x_{1}^{p}x_{2}^{p},x_{1}^{pm}+x_{2}^{pm}\in J_{c}$.
\end{theorem}

\begin{proof}
We re-write these polynomials as $(x_{1}x_{2})^{p}$ and $(x_{1}^{m}+x_{2}^{m})^{p}$. Both are killed by partial derivatives, and it is trivial to check that $x_{1}x_{2}$ and $x_{1}^{m}+x_{2}^{m}$ are $D_{m}$-invariants, so these polynomials are both singular.
\end{proof}

We now consider $\tau$ with $\dim(\tau)>1$: otherwise $L_{c}(\tau)\cong L_{c}(\text{triv})$ (up to a twist and reparametrization). The irreducible representations of $D_{m}$ are two-dimensional, and are denoted $\rho_{a}$, $1\le a<m/2$, where the action of the reflections is given by
\begin{align*}
\begin{bmatrix}
 &\zeta^{-k}\\
\zeta^{k}&  
\end{bmatrix}
\mapsto
\begin{bmatrix}
 &\zeta^{-ak}\\
\zeta^{ak}&  
\end{bmatrix}
\end{align*}

Let $\{e_{1},e_{2}\}$ be a basis for $\tau$.

\begin{theorem}
Let $\tau=\rho_{a}$, with $a>p$, and suppose $m$ is odd. Then, $x_{1}^{p}\otimes e_{1},x_{1}^{p}\otimes e_{2},x_{2}^{p}\otimes e_{1},x_{2}^{p}\otimes e_{2}\in J_{c}$.
\end{theorem}

\begin{proof}
The partial derivatives of these vectors are zero, so the Dunkl operator $D_{1}$ acts by
\begin{center}
$-\displaystyle\sum_{k=0}^{m-1}c_{s}\dfrac{1}{x_{1}-\zeta^{k}x_{2}}(1-s)\otimes s$.
\end{center}

Since $m$ is odd, $c_{s}$ is a constant $c$. We now compute
\begin{align*}
D_{1}(x_{1}^{p}\otimes e_{1})&=-\displaystyle\sum_{k=0}^{m-1}c\dfrac{1}{x_{1}-\zeta^{k}x_{2}}(x_{1}^{p}-(\zeta^{k}x_{2})^{p})\otimes \zeta^{ak}e_{2}\\
&=-\displaystyle\sum_{k=0}^{m-1}c\zeta^{ak}\left(\displaystyle\sum_{\ell=0}^{p-1}\zeta^{\ell k}x_{1}^{p-1-\ell}x_{2}^{\ell}\right)\otimes e_{2}.
\end{align*}

Consider the $x_{1}^{p-1-\ell}x_{2}^{\ell}$ coefficient in the first component: we wish to show that it is zero. This coefficient is
\begin{equation*}
-\displaystyle\sum_{k=0}^{m-1}c\zeta^{ak}\cdot\zeta^{k\ell}=-\displaystyle\sum_{k=0}^{m-1}c(\zeta^{a+\ell})^{k},
\end{equation*}

Note that $0<a+\ell<a+p<a+a<m$. Thus, $m\nmid(a+\ell)$, and $\zeta^{(a+\ell)}\neq1$. Our sum is thus the sum of $d$-th roots of unity for some $d>1$, where $d|m$, each $d$-th root of unity appearing $m/d$ times; the sum is thus equal to zero. In a similar way, we compute
\begin{equation*}
D_{1}(x_{1}^{p}\otimes e_{2})=-\displaystyle\sum_{k=0}^{m-1}c\zeta^{-ak}\left(\displaystyle\sum_{\ell=0}^{p-1}\zeta^{k\ell}x_{1}^{p-1-\ell}x_{2}^{\ell}\right)\otimes e_{1},
\end{equation*}
and we wish to check that
\begin{equation*}
\displaystyle\sum_{k=0}^{m-1}c(\zeta^{-a+\ell})^{k}=0.
\end{equation*}

Again, we have $\zeta^{-a+\ell}\neq1$ because $-a+\ell<-p+\ell<0$, and furthermore it is clear that $|-a+\ell|\le a+\ell<m$. Next,
\begin{align*}
D_{1}(x_{2}^{p}\otimes e_{1})&=-\displaystyle\sum_{k=0}^{m-1}c\dfrac{-\zeta^{-k}}{x_{2}-\zeta^{-k}x_{1}}(x_{2}^{p}-(\zeta^{-k}x_{1})^{p})\otimes\zeta^{ak}e_{2}\\
&=\displaystyle\sum_{k=0}^{m-1}c\zeta^{ak-k}\left(\displaystyle\sum_{\ell=0}^{p-1}\zeta^{-k\ell}x_{2}^{p-1-\ell}x_{1}^{\ell}\right)\otimes e_{2},
\end{align*}
and we need to check that
\begin{equation*}
\displaystyle\sum_{k=0}^{m-1}c(\zeta^{a-\ell-1})^{k}=0,
\end{equation*}
which follows from the fact that $m>a>a-\ell-1>p-(p-1)-1=0$. Finally,
\begin{equation*}
D_{1}(x_{2}^{p}\otimes e_{2})=0
\end{equation*}
is equivalent to
\begin{equation*}
\displaystyle\sum_{k=0}^{m-1}c(\zeta^{-a-\ell-1})^{k}=0,
\end{equation*}
which in turn follows from $0>-a-\ell-1\ge-a-(p-1)-1>-2a>-m$. In exactly the same way, $D_{2}$ kills all four vectors, so we're done.
\end{proof}

\begin{rmk}
The above proof fails for $a=p$ because in the third case, we can have $a-\ell-1=p-(p-1)-1=0$.
\end{rmk}

\begin{theorem}
Let $\tau=\rho_{p}$, and suppose $m$ is odd. Then, $x_{1}^{p}\otimes e_{1},x_{2}^{p}\otimes e_{2},x_{1}^{3p}\otimes e_{2},x_{2}^{3p}\otimes e_{1}\in J_{c}$.
\end{theorem}

\begin{proof}
We can check that the first two generators are killed by Dunkl operators using the same logic as in the previous theorem, noting that in these two cases the necessary strict inequalities still hold. To check that the other two vectors are in $J_{c}$, it is enough to show that that applying Dunkl operators gives multiples of $x_{1}^{p}\otimes e_{2},x_{2}^{p}\otimes e_{1}\in J_{c}$.

We see that
\begin{align*}
D_{1}(x_{1}^{3p}\otimes e_{2})=-\displaystyle\sum_{k=0}^{m-1}c\zeta^{-pk}\left(\displaystyle\sum_{\ell=0}^{3p-1}\zeta^{k\ell}x_{1}^{3p-1-\ell}x_{2}^{\ell}\right)\otimes e_{1},
\end{align*}
and the coefficient on the $x_{1}^{3p-1-\ell}x_{2}^{\ell}$ term in the first component is
\begin{align*}
\displaystyle\sum_{k=0}^{m-1}c(\zeta^{-p+\ell})^{k}.
\end{align*}

We claim that $D_{1}(x_{1}^{3p}\otimes e_{2})$ is a multiple of $x_{1}^{p}\otimes e_{1}\in J_{c}$. To check this, it suffices to show that if the above coefficient is non-zero, then $3p-1-\ell\ge p$. Clearly, $-p+\ell\ge-p>-m$, and also $-p+l\le -p+3p-1=2p-1<m$. Now, if the coefficient on $x_{1}^{3p-1-\ell}x_{2}^{\ell}$ is non-zero, we have $\zeta^{-p+\ell}=1$, and $m|(-p+\ell)$. Therefore, we must have $\ell=p$, and indeed $3p-1-l=2p-1\ge p$. It follows that $x_{1}^{3p}\otimes e_{2}\in J_{c}$ because $x_{1}^{p}\otimes e_{1}\in J_{c}$.

Finally, we compute
\begin{align*}
D_{2}(x_{2}^{3p}\otimes e_{1})=\displaystyle\sum_{k=0}^{m-1}c\zeta^{pk-k}\left(\displaystyle\sum_{\ell=0}^{3p-1}\zeta^{-k\ell}x_{2}^{3p-1-\ell}x_{1}^{\ell}\right)\otimes e_{2},
\end{align*}
and the $x_{2}^{3p-1-\ell}x_{1}^{\ell}$ coefficient is
\begin{align*}
\displaystyle\sum_{k=0}^{m-1}c(\zeta^{p-\ell-1})^{k}.
\end{align*}

We claim that $D_{2}(x_{2}^{3p}\otimes e_{1})$ is a multiple of $x_{2}^{p}\otimes e_{2}\in J_{c}$. Note that $p-\ell-1\ge p-(3p-1)-1=-2p>-m$, and $p-\ell-1\le p-1<m$. Thus, if our coefficient is non-zero, we must have $\zeta^{p-\ell-1}=1$ and $\ell=p-1$, so that $3p-1-\ell=2p>p$. It follows that $x_{2}^{3p}\otimes e_{1}\in J_{c}$ because $x_{2}^{p}\otimes e_{2}\in J_{c}$, and the proof is complete.
\end{proof}

\begin{rmk}
Assuming that $J_{c}$ is indeed generated by the aforemetioned singular vectors in each of the three cases below (which we conjecture to be the case based on computational evidence), we get the following Hilbert series:
\begin{align*}
\tau \text{ trivial: } h_{L_{c}}(t)&=\left(\dfrac{1-t^{p}}{1-t}\right)\left(\dfrac{1-t^{pm}}{1-t}\right)\\
\tau=\rho_{a}, a>p: h_{L_{c}}(t)&=2\left(\dfrac{1-t^{p}}{1-t}\right)^{2}\\
\tau=\rho_{p}: h_{L_{c}}(t)&=2\left(\dfrac{1-t^{p}}{1-t}\right)\left(\dfrac{1-t^{3p}}{1-t}\right)
\end{align*}
\end{rmk}

\section{$S_{n}$ with $p|n$ and $c$ generic}\label{symmetricgeneric}

\begin{theorem}
Consider $G=S_{n}$, where $n$ is even, with $\tau$ trivial and $c$ generic, and $p=2$. Then, $M_{c}/J_{c}$ is a complete intersection with $J_{c}$ generated by $n-1$ elements of degree $2$ and one of degree $4$. Thus, $h_{L_{c}}(t)=(1+t)^{n}(1+t^{2})$.
\end{theorem}

\begin{proof}
Write $J=J_{c}$. Given $i<j$, define $f_{ij} = c(x_i+x_j)(\sum_k x_k) + x_i^2 + x_j^2$ and let $g = \sum_i x_i^2$. Let $I$ be the ideal generated by $\{g, f_{1,2}, \dots, f_{1,n-1}, x_1^4\}$. We claim that $J = I$.

We first check that $f_{ij},g$ are killed by Dunkl operators, so that $g \in J$ and $f_{i,j} \in J$ for all $i,j$. Clearly, $g$ is killed by Dunkl operators, since all of its partial derivatives are zero and it is a $S_{n}$-invariant. Now, consider $D_{1}$ applied to the $f_{i,j}$. When $i,j>1$, we have that $D_{1}f_{i,j}$ is equal to
\begin{center}
$c(x_{i}+x_{j})-c\left(\dfrac{1}{x_{1}-x_{i}}(c(x_{i}-x_{1})\left(\displaystyle\sum_k x_k\right)+x_i^2-x_1^2)+\dfrac{1}{x_{1}-x_{j}}(c(x_{j}-x_{1})\left(\displaystyle\sum_k x_k\right)+x_j^2-x_1^2)\right)$,
\end{center}
which we see vanishes in characteristic 2. Also
\begin{center}
$D_{1}f_{1,j}=c(x_{1}+x_{j})+c\left(\displaystyle\sum_k x_k\right)-c\displaystyle\sum_{\ell\neq1,j}\dfrac{1}{x_{1}-x_{l}}\left(c\left(\displaystyle\sum_k x_k\right)(x_1-x_\ell)+x_\ell^2-x_1^2\right)=0$,
\end{center}
since $n$ is even, so it follows that the $f_{i,j}$ are singular. Note that  $\{g,f_{i,j}\}$ is linearly dependent and one possible basis is $S=\{g, f_{1,2}, \dots, f_{1,n-1}\}$. Linear independence of $S$ follows immediately from noting that for $i=2,\ldots,n-1$, the only appearance of an $x_{i}^{2}$ term in a linear combination of the elements of $S$ is in $f_{1,i}$, and the only appearance of an $x_{n}^{2}$ term is in $g$. To check that $S$ spans our ideal, note that $(c+1)g+f_{1,2}+\cdots+f_{1,n-1}=f_{1,n}$ and $f_{1,i}+f_{1,j}=f_{i,j}$.

Now, note that $(x_i+x_j)^3\in J$, since
\begin{center}
$(x_i+x_j)^3 = c^2(x_i+x_j) g + (x_i+x_j + c(\sum_k x_k)) f_{i,j}$
\end{center}
Also, when $k\neq1$,
\begin{align*}
D_{k}x_{1}^{4}=-c\dfrac{x_{1}^{4}-x_{k}^{4}}{x_{k}-x_{1}}=c(x_{1}+x_{k})^{3},
\end{align*}
and furthermore
\begin{align*}
D_{1}x_{1}^{4}=c\displaystyle\sum_{k\neq1}\dfrac{x_{1}^{4}-x_{k}^{4}}{x_{1}-x_{k}}=c\displaystyle\sum_{k\neq1}(x_{1}+x_{k})^{3},
\end{align*}
so it follows that $x_1^4 \in J$. So we have shown that $I \subseteq
J$. Furthermore, we see that $x_i^4 \in I$ for all $i$ since we have
\begin{center}
$x_1^4 + x_i^4 = (x_1+x_i)(x_1+x_i)^3$.
\end{center}

This implies that $A/I$ is a finite-dimensional vector space over ${\bf
  F}_2(c)$: for example, any monomial of degree at least $3n+1$ must
have some variable $x_i$ raised to at least the 4th power, so it is
divisible by $i$ and must belong to $I$. Hence $A/I$ can only be
nonzero for degrees at most $3n$. Since $A$ has Krull dimension $n$,
and $I$ has $n$ generators, we conclude that $I$ is a complete
intersection, and from the degrees, its Hilbert series is
$h_{A/I}(t) = (t+1)^n (t^2+1)$. Since $I \subseteq J$, we know that $h_{A/I}(t) \ge
h_{A/J}(t)$ coefficientwise. By \cite[Proposition 3.3]{chen}, $h_{A/J}(t) = (t+1)^n
h(t^2)$ for some $h$. So the only possibilities are $h=1+t$, in
which case $I = J$, or $h(t) = 1$ (the case $h(t) = t$ is not
allowed since $h_{A/J}(0) = 1$). If $h(t) = 1$, then $J_2$ contains
$n$ linearly independent polynomials.

To finish, it suffices to check that $J$ only contains $n-1$ linearly independent polynomials of degree
2. Suppose not. When $c=0$, $J_2$ is spanned by $\{x_1^2, \dots, x_n^2
\}$. By considering the limit $c\to 0$, we see that $J$ contains a
generator of the form $\phi = \sum_i \alpha_i(c) x_i^2 + c(\sum_{i \ne
  j} a_{i,j}(c) x_ix_j)$ where $\alpha_i(c)$ and $a_{i,j}(c)$ are
polynomials in $c$ with $\alpha_1(0) = 1$ and $\alpha_j(0) = 0$ for
$j>1$, and we take $a_{i,j}(c)=a_{j,i}(c)$. However, note that
\begin{align*}
D_1\phi&=c\displaystyle\sum_{j\neq1}\alpha_{1,j}x_{j}-c\displaystyle\sum_{j\neq1}\dfrac{1}{x_{1}-x_{j}}\left((\alpha_{1}-\alpha_{j})(x_{1}^{2}-x_{j}^{2})+c\displaystyle\sum_{\ell\neq1,j}(\alpha_{1,\ell}-\alpha_{\ell,j})x_{\ell}(x_{1}-x_{j})\right)\\
&=c\displaystyle\sum_{j\neq1}\left(\alpha_{1,j}x_{j}-(\alpha_{1}-\alpha_{j})(x_{1}+x_{j})-c\displaystyle\sum_{\ell\neq1,j}(\alpha_{1,\ell}-\alpha_{\ell,j})x_{\ell}\right)\\
&=c\left(\displaystyle\sum_{j\neq1}\left(\alpha_{1,j}+\alpha_{1}+\alpha_{j}+c\displaystyle\sum_{\ell\neq1,j}\alpha_{\ell,j}\right)x_{j}+\left(\displaystyle\sum_{j}\alpha_{j}\right)x_{1}\right).
\end{align*}
If $\phi\in J$, we have that the above is identically zero, and since $c$ is indeterminate,
\begin{center}
$\displaystyle\sum_{j\neq1}\left(\alpha_{1,j}+\alpha_{1}+\alpha_{j}+c\displaystyle\sum_{\ell\neq1,j}\alpha_{\ell,j}\right)x_{j}+\left(\displaystyle\sum_{j}\alpha_{j}\right)x_{1}$
\end{center}
must also be identically zero. However, the coefficient on the $x_{1}$ term evaluates to 1 when $c=0$, since $\alpha_{1}(0)=1$ and $a_{j}(0)=0$ for $j>1$. Therefore, $\phi$ is not singular, and we have reached a contradiction. This completes the proof.
\end{proof}

The conjectured generalization of the preceding theorem is the following:
\begin{conjecture}
Consider $G=S_{n}$, where $p|n$, with $\tau$ trivial and $c$ generic. Then, $M_{c}/J_{c}$ is a complete intersection with $J_{c}$ generated by $n-1$ elements of degree $p$ and one of degree $p^{2}$. 
\end{conjecture}

We may write the generators of degree $p$ as $F=F_{0}+cF_{1}+c^{2}F_{2}+\cdots$, where the $F_{i}$ are degree $p$ polynomials individually killed by the Dunkl operators $D_{k}=\partial_{k}-cB_{k}$, where $\partial_{k}$ denotes partial differentiation, and we take $F$ to be in the polynomial ring of the $x_{i}$ with the coefficient field being Laurent series in $c$. Applying Dunkl operators to $F$ and setting the result equal to 0, we get:
\begin{align*}
\partial_{k}F_{0}&=0\\
\partial_{k}F_{m}&=B_{k}F_{m-1}
\end{align*}

The first relation gives us that $F_{0}$ must be of the form $F_{0}=\displaystyle\sum_{i=1}^{n}a_{i}x_{i}^{p}$, where $a_{1},a_{2},\ldots,a_{n}\in\bF_{p}$.

\begin{rmk}
Note that given $F_{m-1}$, there exists $F_{m}$ such that $\partial_{k}F_{m}=B_{k}F_{m-1}$ for $k=1,2,\ldots,n$ if and only if $\partial_{i}B_{k}F_{m-1}=\partial_{k}B_{i}F_{m-1}$ for all $i=1,2,\ldots,n$ (to ensure equality of mixed partial derivatives). 
\end{rmk}

\begin{lemma}
$J_{c}$ contains no generators of degree less than $p$.
\end{lemma}

\begin{proof}
Consider the lowest-degree generators, which must be killed by Dunkl operators. We may still write a generator $F$ in the form $F=F_{0}+cF_{1}+\cdots$ satisfying the same relations as above; in particular, $\partial_{k}F_{0}=0$ for all $k$. However, this is impossible unless $\text{deg}(F)$ is a multiple of $p$.
\end{proof}

\begin{lemma}
$\displaystyle\sum_{i=1}^{n}a_{i}=0$.
\end{lemma}

\begin{proof}
We have
\begin{align*}
B_{1}F_{0}&=\displaystyle\sum_{j=2}^{n}\dfrac{1}{x_{1}-x_{j}}(a_{1}x_{1}^{p}+a_{j}x_{j}^{p}-a_{1}x_{j}^{p}-a_{j}x_{1}^{p})\\
&=\displaystyle\sum_{j=2}^{n}\dfrac{(a_{1}-a_{j})(x_{1}^{p}-x_{j}^{p})}{x_{1}-x_{j}}\\
&=\displaystyle\sum_{j=2}^{n}\displaystyle\sum_{r=0}^{p-1}(a_{1}-a_{j})x_{1}^{r}x_{j}^{p-1-r}=\partial_{1}F_{1}.
\end{align*}
However, note that in characteristic $p$, $\partial_{1}F_{1}$ cannot have terms of the form $ax_{1}^{p-1}$. Thus,
\[
\sum_{j=2}^{n}(a_{1}-a_{j})=-\displaystyle\sum_{j=1}^{n}a_{j}=0\Rightarrow
\sum_{j=i}^{n}a_{i}=0. \qedhere
\]
\end{proof}

\begin{conjecture} \label{conjecture:gens}
Given $F_{0}=\displaystyle\sum_{i=1}^{n}a_{i}x_{i}^{p}$ with $\displaystyle\sum_{j=i}^{n}a_{i}=0$, there exist $F_{1},F_{2},\ldots$ such that $\partial_{k}F_{m}=B_{k}F_{m-1}$ for all positive integers $m$. Furthermore, at each step in the recursion, each $F_{m}$ is unique up to adding $p$-th powers, and the set of all possible $F=\displaystyle\sum_{i=0}^{\infty}c^{i}F_{i}$ forms an $\bF_{p}$-vector space of dimension $p-1$.
\end{conjecture}

We now prove a few parts of this conjecture.

\begin{proposition}
If, given $F_{m-1}$, there exists $F_{m}$ for which $\partial_{k}F_{m}=B_{k}F_{m-1}$, then $F_{m}$ is unique up to adding $p$-th powers.
\end{proposition}

\begin{proof}
Suppose there exists some $F'_{m}$ satisfying $\partial_{k}F'_{m}=B_{k}F_{m-1}$. Then, all partials of $F_{m}-F'_{m}$ vanish, so $F_{m}-F'_{m}$ must be the sum of $p$-th powers.
\end{proof}

\begin{proposition}
Assume that if we take each $F_{i}$ with $i>0$ to include no $p$-th powers, we can construct $F_{1},F_{2},\ldots$. Then, the space of all possible generators $F$ of degree $p$ has dimension $n-1$.
\end{proposition}

\begin{proof}
Let $F(a_{1},a_{2},\ldots,a_{n})\in J_{c}$, where $\displaystyle\sum_{i=1}^{n}a_{i}=0$, be the generator obtained by taking $F_{0}=\displaystyle\sum_{i=1}^{n}a_{i}x_{i}^{p}$ (assuming the first part of \ref{conjecture:gens}). The space of $F_{0}$ is isomorphic to the space of $\bF_{p}$-vectors whose components sum to zero, which has dimension $n-1$. It suffices to show that the $F(a_{1},a_{2},\ldots,a_{n})$ span the space of degree $p$ generators. Suppose that we have a generator $F$ in which we add a sum of $p$-th powers $F'_{0}=\displaystyle\sum_{i=1}^{n}b_{i}x_{i}^{p}$ to $F_{k}$, and $k\ge1$ is minimal. Letting $F_{0}=\displaystyle\sum_{i=1}^{n}a_{i}x_{i}^{p}$, note that $F'=F-F(a_{1},a_{2},\ldots,a_{n})\in J_{c}$. Furthermore, $F'c^{-k}\in J_{c}$, and $F'$ has the form $\displaystyle\sum_{i=0}^{\infty}b^{i}F'_{i}$, and we may iterate the argument on $F'$. It follows that $F$ is a linear combination of the $F(a_{1},a_{2},\ldots,a_{n})$, where we may take our weights to be appropriate powers of $c$.
\end{proof}

\begin{rmk}
We have now reduced \ref{conjecture:gens} to proving the existence of the $F_{i}$, provided that we include no $p$-th powers in any of $F_{1},F_{2},\ldots$. Furthermore, if we can make this construction in the field of Laurent series, the series can be replaced by rational functions since the number and degrees of the generators of $J$ doesn't change if we enlarge the field from rational functions to Laurent series. This follows from the fact that the matrices for $\beta$ have entries in the field of rational functions in $c$, and the rank of a matrix doesn't change if the field of coefficients is enlarged.
\end{rmk}

\begin{proposition}
Take $p\neq2$. If, at the respective steps, we include no $p$-th powers, we have
\begin{center}
$F_{1}=-\displaystyle\sum_{1\le i<j\le n}\displaystyle\sum_{\substack{r,s>0\\r+s=p}}\dfrac{ra_{i}+sa_{j}}{rs}x_{i}^{r}x_{j}^{s}$
\end{center}
and
\begin{center}
$F_{2}=\displaystyle\sum_{i<j<k}\displaystyle\sum_{\substack{r,s,t>0\\r+s+t=p}}\left(\dfrac{ra_{i}+sa_{j}+ta_{k}}{rst}\right)x_{i}^{r}x_{j}^{s}x_{k}^{t}+\displaystyle\sum_{i<j}\displaystyle\sum_{\substack{r,s>0\\r+s=p}}\dfrac{a_{i}-a_{j}}{r}\left(\dfrac{1}{s}-2\displaystyle\sum_{d=1}^{s}\dfrac{1}{d}\right)x_{i}^{r}x_{j}^{s}$.
\end{center}
\end{proposition}

\begin{proof}
We already have
\begin{center}
$\partial_{1}F_{1}=\displaystyle\sum_{r=0}^{p-1}(a_{1}-a_{j})x_{1}^{r}x_{j}^{p-1-r}$,
\end{center}
and similar relations hold for the other partial derivatives. We find that the system of differential equations is indeed satisfied by
\begin{center}
$F_{1}=-\displaystyle\sum_{1\le i<j\le n}\displaystyle\sum_{\substack{r,s>0\\r+s=p}}\dfrac{ra_{i}+sa_{j}}{rs}x_{i}^{r}x_{j}^{s}$,
\end{center}
and it is clear that $F_{1}$ must be unique (since it is homogeneous). Continuing,
\begin{align*}
B_{1}F_{1}&=\displaystyle\sum_{j=2}^{n}\dfrac{1}{x_{1}-x_{j}}\left(-\displaystyle\sum_{1\le i<k\le n}\displaystyle\sum_{\substack{r,s>0\\r+s=p}}(1-s_{ij})\dfrac{ra_{i}+sa_{k}}{rs}x_{i}^{r}x_{k}^{s}\right)\\
&=-\displaystyle\sum_{j=2}^{n}\dfrac{1}{x_{1}-x_{j}}(1-s_{ij})\left(\displaystyle\sum_{k\neq1,j}\displaystyle\sum_{\substack{r,s>0\\r+s=p}}\dfrac{ra_{1}+sa_{k}}{rs}x_{1}^{r}x_{k}^{s}+\displaystyle\sum_{k\neq1,j}\displaystyle\sum_{\substack{r,s>0\\r+s=p}}\dfrac{ra_{j}+sa_{k}}{rs}x_{j}^{r}x_{k}^{s}\right)\\
&-\displaystyle\sum_{j=2}^{n}\dfrac{1}{x_{1}-x_{j}}(1-s_{ij})\left(\displaystyle\sum_{\substack{r,s>0\\r+s=p}}\dfrac{ra_{1}+sa_{j}}{rs}x_{1}^{r}x_{j}^{s}\right).
\end{align*}

Consider the first term. We have
\begin{align*}
&-\displaystyle\sum_{j=2}^{n}\dfrac{1}{x_{1}-x_{j}}(1-s_{ij})\left(\displaystyle\sum_{k\neq1,j}\displaystyle\sum_{\substack{r,s>0\\r+s=p}}\dfrac{ra_{1}+sa_{k}}{rs}x_{1}^{r}x_{k}^{s}+\displaystyle\sum_{k\neq1,j}\displaystyle\sum_{\substack{r,s>0\\r+s=p}}\dfrac{ra_{j}+sa_{k}}{rs}x_{j}^{r}x_{k}^{s}\right)\\
&=\displaystyle\sum_{j=2}^{n}\left(-\displaystyle\sum_{k\neq1,j}\displaystyle\sum_{\substack{r,s>0\\r+s=p}}\dfrac{ra_{1}+sa_{k}}{rs}x_{k}^{s}\cdot\dfrac{x_{1}^{r}-x_{j}^{r}}{x_{1}-x_{j}}+\displaystyle\sum_{k\neq1,j}\displaystyle\sum_{\substack{r,s>0\\r+s=p}}\dfrac{ra_{j}+sa_{k}}{rs}x_{k}^{s}\cdot\dfrac{x_{1}^{r}-x_{j}^{r}}{x_{1}-x_{j}}\right)\\
&=\displaystyle\sum_{j\neq1}\left(-\displaystyle\sum_{k\neq1,j}\displaystyle\sum_{\substack{s>0\\a+b+s=p-1}}\dfrac{(a+b+1)a_{1}+sa_{k}}{(a+b+1)s}x_{k}^{s}x_{1}^{a}x_{j}^{b}+\displaystyle\sum_{k\neq1,j}\displaystyle\sum_{\substack{s>0\\a+b+s=p-1}}\dfrac{(a+b+1)a_{j}+sa_{k}}{(a+b+1)s}x_{k}^{s}x_{1}^{a}x_{j}^{b}\right)\\
&=-\displaystyle\sum_{j\neq1}\displaystyle\sum_{k\neq1,j}\displaystyle\sum_{\substack{s>0\\a+b+s=p-1}}\dfrac{a_{1}-a_{j}}{s}x_{k}^{s}x_{1}^{a}x_{j}^{b}\\
&=-\displaystyle\sum_{1<j<k}\displaystyle\sum_{\substack{s,t>0\\r+s+t=p-1}}\left(\dfrac{a_{1}-a_{j}}{t}+\dfrac{a_{1}-a_{k}}{s}\right)x_{1}^{r}x_{j}^{s}x_{k}^{t}-\displaystyle\sum_{j\neq1}\displaystyle\sum_{k\neq1,j}\displaystyle\sum_{\substack{t>0\\r+t=p-1}}\dfrac{a_{1}-a_{j}}{t}x_{1}^{r}x_{k}^{t}\\
&=\displaystyle\sum_{1<j<k}\displaystyle\sum_{\substack{s,t>0\\r+s+t=p-1}}\left(\dfrac{(r+1)a_{1}+sa_{j}+ta_{k}}{st}\right)x_{1}^{r}x_{j}^{s}x_{k}^{t}+\displaystyle\sum_{k\neq1}\displaystyle\sum_{\substack{t>0\\r+t=p-1}}\dfrac{a_{1}-a_{k}}{t}x_{1}^{r}x_{k}^{t}.
\end{align*}

Now, for the other term (in the last step we re-index from $j$ to $k$ for convenience),
\begin{center}
$-\displaystyle\sum_{j=2}^{n}\dfrac{1}{x_{1}-x_{j}}(1-s_{ij})\left(\displaystyle\sum_{\substack{r,s>0\\r+s=p}}x_{1}^{r}x_{j}^{s}\right)=-2\displaystyle\sum_{k=2}^{n}\displaystyle\sum_{\substack{r>s>0\\r+s=p}}\dfrac{ra_{1}+sa_{k}}{rs}\cdot\dfrac{x_{1}^{r}x_{k}^{s}-x_{1}^{s}x_{k}^{r}}{x_{1}-x_{k}}$.
\end{center}

We calculate the coefficient on the term $x_{1}^{b}x_{k}^{d}$ where $b+d=p-1$, upon combining these two parts. The first part gives a coefficient of $(a_{1}-a_{k})\cdot\dfrac{1}{d}$. For the second part, we get a contribution of $-2\cdot\dfrac{ra_{1}+sa_{k}}{rs}=-2\cdot\dfrac{a_{1}-a_{k}}{s}$ if and only if $b,d\ge s$, which happens for $s=1,2,\ldots,\min(b,d)$. Our coefficient is thus
\begin{center}
$(a_{1}-a_{k})\left(\dfrac{1}{d}-2\left(\dfrac{1}{1}+\dfrac{1}{2}+\cdots+\dfrac{1}{\min(b,d)}\right)\right)$.
\end{center}

When $p$ is odd, we see that in fact both values of $\min(b,d)$ give the same value since $d=p-(b+1)$. Thus (after re-indexing again), we have
\begin{center}
$B_{k}F_{1}=\displaystyle\sum_{1<j<k}\displaystyle\sum_{\substack{s,t>0\\r+s+t=p-1}}\left(\dfrac{(r+1)a_{1}+sa_{j}+ta_{k}}{st})x_{1}^{r}x_{j}^{s}x_{k}^{t}+\displaystyle\sum_{j\neq1}\displaystyle\sum_{\substack{s>0\\r+s=p-1}}(a_{1}-a_{j})\right)\left(\dfrac{1}{s}-2\displaystyle\sum_{i=1}^{s}\dfrac{1}{i}\right)x_{1}^{r}x_{j}^{s}$,
\end{center}
so taking a $x_{1}$-antiderivative and a (modified, as before) symmetric sum gives
\begin{center}
$F_{2}=\displaystyle\sum_{i<j<k}\displaystyle\sum_{\substack{r,s,t>0\\r+s+t=p}}\left(\dfrac{ra_{i}+sa_{j}+ta_{k}}{rst}\right)x_{i}^{r}x_{j}^{s}x_{k}^{t}+\displaystyle\sum_{i<j}\displaystyle\sum_{\substack{r,s>0\\r+s=p}}\dfrac{a_{i}-a_{j}}{r}\left(\dfrac{1}{s}-2\displaystyle\sum_{d=1}^{s}\dfrac{1}{d}\right)x_{i}^{r}x_{j}^{s}$,
\end{center}
as desired. 
\end{proof}

\begin{proposition}
When $p=3$, we can take
\begin{center}
$F=\displaystyle\sum_{i}a_{i}x_{i}^{3}-c\displaystyle\sum_{i,j}(a_{i}-a_{j})x_{i}^{2}x_{j}+c^{2}\left(\displaystyle\sum_{i<j<k}(a_{i}+a_{j}+a_{k})x_{i}x_{j}x_{k}+\displaystyle\sum_{i,j}(a_{i}-a_{j})x_{i}^{2}x_{j}-\displaystyle\sum_{i}a_{i}x_{i}^{3}\right)$.
\end{center}
\end{proposition}

\begin{proof}
The values of $F_{0},F_{1}$ agree with what has already been checked. However, in $F_{2}$, note that we add a sum of $p$-th powers, namely, $-\displaystyle\sum_{i}a_{i}x_{i}^{3}$: it can be checked directly here that $B_{k}F_{2}=0$ for all $k$.
\end{proof}

\begin{rmk}
While we conjecture that we could have taken $F_{2}$ to contain no $p$-th powers and continued the recursive process to generate $F_{3},F_{4},\ldots$ infinitely, note that adding the $p$-th powers, in this case, terminated the process.
\end{rmk}

\section{Acknowledgments}
\indent\indent This research was conducted with the support of the Program in Research in Mathematics, Engineering, and Science (PRIMES), the MIT Undergraduate Research Opportunities Program (UROP), and the John Reed fund. The author thanks Steven Sam and Pavel Etingof for supervising the project, as well as Sheela Devadas for helpful discussions. In addition, the author thanks the anonymous referees for their comments.

Computational data for this project were gathered in SAGE (\cite{sage}).


\begin{thebibliography}{99}

\setlength{\itemsep}{-1mm}
\small

\bibitem[BC]{chen} Martina Balagovi\'c, Harrison Chen, Representations
  of rational Cherednik algebras in positive characteristic, preprint, \arxiv{1107.0504v1}.


\bibitem[BC]{bc} Yuri Berest and Oleg Chalykh, Quasi-invariants of Complex Reflection Groups, {\it Compositio Math.} {\bf 147} (2011), no.~3, 965-1002, \arxiv{0912.4518}

\bibitem[BFG]{bfg} Roman Bezrukavnikov, Michael Finkelberg, and Victor
  Ginzburg, Cherednik algebras and Hilbert schemes in characteristic
  $p$, with appendices by Pavel Etingof and Vadim Vologodsky, {\it
    Represent. Theory} {\bf 10} (2006), 254--298,
  \arxiv{math/0312474v5}.

\bibitem[BO]{bo} Roman Bezrukavnikov and Andrei Okounkov, work in progress 
  
\bibitem[CE]{chmutova} Tatyana Chmutova and Pavel Etingof, On some representations of the rational Cherednik algebra, {\it Represent. Theory} {\bf 7} (2003), 641--650, \arxiv{math/0303194v2}




\bibitem[E]{etingof} Pavel Etingof, Reducibility of the polynomial
  representation of the degenerate double affine Hecke algebra,
  \arxiv{0706.4308v1}.

\bibitem[EG]{calogeromoser} Pavel Etingof and Victor Ginzburg,
  Symplectic reflection algebras, Calogero-Moser space, and deformed
  Harish-Chandra homomorphism, {\it Invent. Math.} {\bf 147} (2002),
  no.~2, 243--348, \arxiv{math/0011114v6}.

\bibitem[EM]{etingofma} Pavel Etingof and Xiaoguang Ma, Lecture notes
  on Cherednik algebras, \arxiv{1001.0432v4}.


\bibitem[Gor]{gordon} Iain Gordon, Baby Verma modules for rational
  Cherednik algebras, {\it Bull. London Math. Soc.} {\bf 35} (2003),
  no.~3, 321--336, \arxiv{math/0202301v3}.


\bibitem[Lat]{latour} Fr\'ed\'eric Latour, Representations of rational
  Cherednik algebras of rank 1 in positive characteristic, {\it J. Pure Appl. Algebra} {\bf 195} (2005), no.~1, 97--112, \arxiv{math/0307390v1}.



\newcommand{\etalchar}[1]{$^{#1}$}
\bibitem[S{\etalchar{+}}09]{sage}
W.\thinspace{}A. Stein et~al., \emph{{S}age {M}athematics {S}oftware ({V}ersion
  4.6.1)}, The Sage Development Team, 2011, \url{http://www.sagemath.org}.


\end{thebibliography}
\end{document}